\documentclass[11pt, twoside, a4paper]{article}

\usepackage{amsmath}     				
\usepackage{amssymb}   					
\usepackage{amsthm}     				
\usepackage{graphicx}    				

\usepackage{textcomp}    				
\usepackage[T1]{fontenc} 				
\usepackage{marvosym}    				
\usepackage[sc]{mathpazo}  	 			

\usepackage{authblk} 					
\usepackage[usenames]{xcolor}  			
\usepackage{lastpage} 					

\usepackage{enumerate}	 				

\usepackage{hyperref} 					

\definecolor{ForestGreen}{rgb}{0.15,0.416,0.18}
\definecolor{EgyptBlue}{rgb}{0.063,0.2,0.65}
\hypersetup{
	colorlinks=true,
	linkcolor=EgyptBlue,         
   citecolor=EgyptBlue,          
   urlcolor=ForestGreen          
}

\newtheorem{theorem}{Theorem}[section]
\newtheorem{corollary}[theorem]{Corollary}
\newtheorem{lemma}[theorem]{Lemma}
\newtheorem{proposition}[theorem]{Proposition}

\theoremstyle{definition}

\theoremstyle{definition}
\newtheorem{remark}[theorem]{Remark}
\theoremstyle{definition}

\numberwithin{equation}{section}
\numberwithin{table}{section}
\numberwithin{figure}{section}

\title{On prescribed energy saddle-point solutions to indefinite problems}

\author[1,2]{\textbf{Yavdat Il'yasov}}
\author[2]{\textbf{Edcarlos D. Silva}}
\author[2]{\textbf{Maxwell L. Silva}}

\affil[1]{Institute of Mathematics, Ufa Federal Research Centre, RAS,  450008, Ufa, Russia}
\affil[2]{Department of Mathematics, Federal University of Goi\'as 74001-970, Goi\^ania - GO, Brazil}

\newcommand{\doi}[1]{\url{https://doi.org/#1}}
\newcommand{\MR}[1]{\href{https://www.ams.org/mathscinet-getitem?mr=#1}{MR#1}}
\newcommand{\ZB}[1]{\href{https://zbmath.org/?q=an:#1}{Zbl~#1}}

\makeatletter
\renewcommand{\maketitle}{\bgroup\setlength{\parindent}{0pt}
\vspace{1truecm}
\begin{center}{\vbox{\titlefont\@title}}\end{center}
\vspace{0.5truecm}
\begin{center}{\@author} \end{center}

\egroup
}

\renewcommand{\@fnsymbol}[1]{%
    \ifcase#1 \or {\,\Letter\!} \or\textasteriskcentered\or \textasteriskcentered\textasteriskcentered 
    \else\@ctrerr\fi}
\makeatother

\newcommand*{\titlefont}{\fontsize{18}{21.6}\selectfont\textbf}

\makeatother

\begin{document}

\maketitle

\pagestyle{plain}

\begin{center}
\noindent
\begin{minipage}{0.85\textwidth}\parindent=15.5pt

%

{\small{
\noindent {\bf Abstract.} A minimax variational principle for  saddle-point solutions with prescribed energy levels is introduced. The approach is based on the development of the linking theorem to the energy level nonlinear generalized Rayleigh quotients. An application to indefinite elliptic Dirichlet problems is presented. Among the
consequences, the existence of solutions with zero-energy levels is obtained. }
\smallskip

\noindent {\bf{Keywords:}} Indefinite problems, Linking Theorems, Rayleigh quotient.
\smallskip

\noindent{\bf{2020 Mathematics Subject Classification:}} 35G15, 35G20, 35G25, 35G30.
}

\end{minipage}
\end{center}

\section{Introduction}
The Mountain Pass Theorem  introduced by  Ambrosetti \& Rabinowitz \cite{AmbR}  and its generalization  as the  Benci \& Rabinowitz Linking Theorem \cite{BencRab} is a  powerful tool  to establish the existence of solutions for nonlinear differential problems of the variational form$$DE_\mu(u)=0,$$where $E_\mu(u)$ is an energy functional defined on a Banach space $W$. The solutions obtained by this method usually correspond to saddle critical points of the energy functional and are often referred to as mountain pass-type solutions or saddle-point (SP for short) solutions.   In essence, this method is topological, which makes it possible to use it for solving problems of very general forms. On the other hand, this generality often makes it difficult to  find out detailed information about the obtained solutions.

In the present note, we address the following question:\textit{ Is it possible to determine the level of the energy functional $E=E_\mu(\hat{u})$ for the saddle-point solution }$\hat{u}$? The question naturally appears in the theory of the mountain pass method itself \cite{CarvIlyaSant,motreanu,struwe}; moreover, the study of the energy levels of solutions is important in many applied problems (see e.g., \cite{Klinkhamer, Manton}).

We propose a method that allows to predict the energy level $E=E_\mu(u)$ of saddle-point solutions. We illustrate it by studying the following indefinite elliptic problem	\cite{BencRab}
\begin{align}\label{p}
		\begin{cases}
				& - \Delta u - \lambda u  =\mu |u|^{q-1}u+g(x,u)~~ \mbox{in}~ \Omega,\\
				&~~u=0~~\mbox{on}~~ \partial \Omega
\end{cases}
\end{align}
Here $\Omega$ is a bounded smooth domain in $\mathbb{R}^N$, $N\geq 1$,  $\lambda \in \mathbb{R}$ and  $\mu>0$, $1 < q < 2$.   
We assume that $g:\Omega \times \mathbb{R} \to \mathbb{R}$ is a Carath\'eodory function with primitive $G(x, u) =\int_0^u g(x, v) dv$ and 
\begin{description}
	\item[$(A^1)$] $g(x, 0) = 0$ and $\lim sup_{u\to 0} \frac{g(x,u)}{u}
\leq 0$, uniformly in $x \in \Omega$;
\item[$(A^2)$] $ \exists \gamma \in (2, 2^*)$, $C>0$ such that $|g(x, u)| \leq C
(1+|u|^{\gamma -1})$;
\item[$(A^3)$] $0<G(x,u)$  a.e. $x \in \Omega$, $\forall u \in \mathbb{R}$, \\
$\exists \alpha>2$, $R_0>0$ such that $\alpha G(x,u)\leq g(x,u)u$, for a.e. $x \in \Omega$, and  $|u|\geq R_0 $. 
\end{description}

The operator  $-\Delta$  with  Dirichlet boundary
conditions is known to define a self-adjoint operator in $L^2(\Omega)$ (see, e.g., \cite{edmund}) and its spectrum consists of an infinite sequence ordered  $0\!<\!\lambda_1\!<\!\lambda_2\!\leq\!  \ldots$ of eigenvalues repeated according to their finite multiplicity.  



By a weak solution of \eqref{p} we mean a critical point
$u \in  W^{1,2}_0(\Omega)$ of the energy functional  
\begin{equation*}
E_\mu(u)=\frac{1}{2} \left(\int |\nabla u|^2 dx -\lambda\int | u|^2 dx\right)- \frac{\mu}{q}\int | u|^q dx-\int G(x,u) dx,\;\;\;\; u \in W:= W^{1,2}_0(\Omega), 
\end{equation*}i.e., $DE_\mu(u)=0$.
 Here and subsequently,  $DE_\mu(u)$ denotes the Fr\'echet derivative of $E_{\mu}$  at $u \in W$ and $DE_\mu(u)(v)$ denotes its the directional derivative  in direction $v  \in W$.
 
In the case $\lambda >\lambda_1$, 
\eqref{p} is called \textit{indefinite problem}, since in this case the linear part of \eqref{p} is  indefinite and   $0$ is a saddle point of the functional $E_{\lambda,\mu}(u)$  \cite{BencRab, motreanu}.

We are interested in a \textit{prescribed energy solution} \cite{Ilyas21_JMS, Ilyas21} of \eqref{p}, i.e.,  $u_\mu \in W\setminus 0$  which for a given $E$ satisfies $
E_\mu(u_\mu)=E ~~\mbox{and}~~ DE_\mu(u_\mu)=0.
$
We seek for such solutions using the so called \textit{energy level Rayleigh quotient}\cite{Ilyas,Ilyas21_JMS, Ilyas21}   
\begin{equation}\label{R}
\mathcal{R}^{E}(u):=\frac{\frac{1}{2}\left(\int |\nabla u|^2 \,dx-\lambda \int | u|^2 \,dx\right)-\int G(x,u)\,dx-E}{\frac{1}{q}\int | u|^q dx}, \, \, u \in W\setminus \{0\},~ E \in \mathbb{R}.
\end{equation}
 Notice that for $u \in W\setminus \{0\},$
\begin{equation}\label{R1}
\begin{array}{cc}
	&E_{\mu}(u)=E\;\Leftrightarrow\;\mu=\mathcal{R}^{E}(u),\\
	&E_{\mu}(u)=E,~~DE_{\mu}(u)=0\;\;\Leftrightarrow\;\;DR^{E}(u)=0,~~\mu=\mathcal{R}^{E}(u).
\end{array}
\end{equation} 
Now, with the convention that $\lambda_0=-\infty,$  our main result is as follows\begin{theorem}\label{thm1}
 Assume that $1 < q < 2 < \gamma < 2^*$, $\lambda \in (\lambda_k, \lambda_{k + 1})$ and $(A^1)$-$(A^3)$ hold. Then there exists $E^k_\lambda>0$ such that  for any given $E\in ( 0, E^k_\lambda)$ corresponds $\mu^k_\lambda(E) \in (0, +\infty)$ such that \eqref{p} with $\mu=\mu^k_\lambda(E)$ possesses a weak solution  $u_{\mu^k_\lambda(E)}$ with energy value $E$, i.e., $DE_{\mu^k_\lambda(E)}(u_{\mu^k_\lambda(E)})=0$, $E_{\mu^k_\lambda(E)}(u_{\mu^k_\lambda(E)})=E$.

Furthermore, 
\begin{description}
	\item[(i)]
	$\mu^k_\lambda(E)$ nonincreasing function in $(0, E^k_\lambda)$;
	\item[(ii)] If $\lambda <\lambda_1$, there exists a limit $\lim_{E\to 0} \mu^0_\lambda(E)=\bar{\mu}_\lambda(0)<+\infty$ such that   \eqref{p} possesses a weak solution  $u_{\bar{\mu}_\lambda(0)}$ with zero energy value $E=0$ and $\mu=\bar{\mu}_\lambda(0)$.
\end{description}
\end{theorem}

To find solutions of \eqref{p}, we apply Mountain Pass Theorem to the energy level Rayleigh quotient $\mathcal{R}^{E}(u)$. Note that $\mathcal{R}^{E} \in C^{1}(W \setminus \{0\}, \mathbb{R})$.  However for the application of the Mountain Pass Theorem, in general, it is required that	the functional belongs to $C^{1}(W, \mathbb{R})$. 
Below we overcome this difficulty by using an appropriate truncation function for $\mathcal{R}^{E}$ which can be properly introduced in the case $E>0$. In the zero-energy case $E=0$, the solution is obtained  by passing to the limit $E\to 0$.

\begin{remark} The zero-energy case $E=0$ is particularly interesting since for linear problems (such as $Lu =\lambda u$, where $L$  is  a self-adjoint
 linear operator
on a Hilbert space $H=W$), any isolated spectral point $\lambda_n$, $n=1,2,\ldots$,  corresponds to an eigenfunction $\phi_n$ of the zero-energy    level, i.\,e., $E=\frac{1}{2}\left\langle  \phi_n,L \phi_n\right\rangle -\lambda_n\frac{1}{2}\left\langle \phi_n,\phi_n\right\rangle=0$.
\end{remark}

\begin{remark} 

The weak solution $u=u_{\mu}\in W^{1,2}_0(\Omega)$ satisfies $-\Delta u= a(x)(1+|u|),$ where $a(x)\in L^{\frac{N}{2}}(\Omega),$ due $(A^1)\!-\!(A^2).$ This yields by elliptic regularity that  $u\in C^{1,\gamma}(\Omega)$(see \cite{struwe} Appendix B). Moreover, if $g$ is Holder continuous, then $u\in C^{2}(\Omega).$

\end{remark}
\begin{remark}
	Observe that $\mu |u|^{q-1}u+g(x,u)$  does not satisfies the Ambrosetti Rabinowitz condition $(A^3)$  \cite{AmbR}. 
\end{remark}
\subsection{Notation}
We shall use the following notations:\\

\noindent i) $||\cdot||_W$ is the norm on the Sobolev Space $W:= W^{1,2}_{0}(\Omega)$ induced by the usual inner product;

\noindent ii) $|u|_{L^r}:=(\int_\Omega |u|^r\, dx)^{1/r}$, $1\leq r<+\infty,$ is the norm on the Lebesgue space $L^r:=L^r(\Omega);$

\noindent iii) $S_{p}$ is the best Sobolev constant for the embedding $W_{0}^{1,2}(\Omega)\subset L^p(\Omega)$, $1\leq p\leq 2^*$;

\noindent iv) $\|\cdot\|_*$ denotes the norm in the dual space $W^*$, 

\noindent v) $d(A,B):=\min\{\|u-v\|_W: ~u \in A,~v \in B\}$ denotes the distance between sets $A,B \in W$.




\subsection{Outline}
The present work is organized as follows. In the forthcoming section we consider the nonlinear Rayleigh quotient for our problem together with its behavior near origin. In Section 3 we prove that the Cerami condition for a suitable functional is verified. Section 4 is devoted to the prove that function given by the Rayleigh quotient is bounded from below. In Section 5 we prove our main results by using the nonlinear Rayleigh quotient together a fine analysis on the associated fibering maps.


\section{Preliminaries}


Let   $(e_k) \subset W$ be the orthogonal basis of the eigenfunctions of $(-\Delta,  W^{1,2}_0(\Omega))$ with zero Dirichlet conditions satisfying $\| e_k\|_W^2 = \lambda_k$ and $|e_k|_{L^2}^2 = 1,$ for each $k \in \mathbb{N}.$  Assume
$\lambda \in (\lambda_k, \lambda_{k+1}),$ fixed for some $k\in\mathbb{N}.$ We may write 
$$
W\;\;=\;\; W^{+}\oplus W^-,\;\;\;\text{where}\;\;\;W^{-}= span \{e_1,e_2, \ldots,e_k\}\;\;\;\;\;  and\;\;\;\;\;W^{+}=\overline{span\{e_{k + 1},e_{k+2},\ldots\}}.
$$	
Then one can use the following equivalent norm $\|\cdot\|_1$ to $\|\cdot\|_W$ in $W$ 
$$
	\|u\|^2_1=\sum_{i=k+1}^\infty (\lambda_i-\lambda) u_i^2+\sum_{i=1}^k(\lambda-\lambda_i) u_1^2 := \|u^+\|_1^2+\|u^-\|_1^2.
$$
where $u_i \;=\;( u,e_k )$. Then for any $u=(u^+ + u^-) \in W,\;u^{\pm}\in W^{\pm}$, $c_0 \|u\|_1^2 \leq \|u\|_W^2 \leq c_1\|u\|_1^2$, where $0<c_0,c_1<+\infty$  do not depend on $u \in W$, and	
\begin{align*}
	H_{\lambda}(u):=&\|u\|_W^2-\lambda|u|_{L^2}^2\;=\;\int |\nabla u|^2 dx -\lambda\int | u|^2 dx\;\;=\\
	&H_{\lambda}(u^+)+H_{\lambda}(u^-)\;\;=\;\;\|u^+\|_1^2-\|u^-\|_1^2,~u \in W.
\end{align*}
Notice that $H_\lambda(u)=-\|u\|_1^2<0$ if $u \in W^{-} \setminus \{0\}$, and  $H_\lambda(u)=\|u\|_1^2>0$ if $u \in W^{+} \setminus \{0\}$, for $\lambda \in (\lambda_k,, \lambda_{k+1}).$ 

In these notations, we have 
\begin{align*}
E_\mu(u)=&\frac{1}{2}H_\lambda(u)- \frac{\mu}{q} |u|_{L^q}^q-\int G(x,u)dx,~~\\
\mathcal{R}^{E}(u)=&\frac{\frac{1}{2}H_\lambda(u)-\int G(x,u)dx-E}{\frac{1}{q}|u|_{L^q}^q},~~ u \in W\setminus \{0\}.
\end{align*}

Obviously, $\mathcal{R}^E \in C^{1}(W \setminus \{0\}, \mathbb{R})$ and 
\begin{equation}\label{R1}
DE_\mu(u)=0~\mbox{and}~~E_\mu(u)=E\,\;\;\;\;\Leftrightarrow\;\;\;\;\;D\mathcal{R}^{E}(u)=0~\mbox{and}~~\mu=\mathcal{R}^{E}(u),~~ u \in W\setminus 0.
\end{equation} 
 To avoid the singularity at origin of $\mathcal{R}^E$, we define $\phi_\rho\in C^{\infty}( \mathbb{R},[0,1])$, for  $\rho>0$  such that 
$$
\left\{\begin{array}{l}\phi_\rho(s)=0\;\;\text{if}\;\;|s|<\rho/2\\\phi_\rho(s)=1\;\;\text{if}\;\;|s|>\rho\end{array}\right.
$$
and introduce
$$
\mathcal{R}^E_\rho(u)=\left\{\begin{aligned}
	&~\phi_\rho(\|u\|_1) \mathcal{R}^E(u), ~u \in  W\setminus 0\\
	&~0,~~ u=0
\end{aligned} \right.
$$
Thus, 	$\mathcal{R}^E_\rho(u) \in C^1(W)$ for any $\rho>0$. 

 We need also the following estimate on the sign of energy near the origin
using the ball  $B_r:=\{u \in W:~\|u\|_1\leq r\},$ $r>0.$ 
\begin{lemma}\label{lemNeg} Assume that $E>0$ and $\lambda \in (\lambda_k,, \lambda_{k+1}).$
Then there exists $\rho(E)>0$ such that $\mathcal{R}^E(u)<0$ for $u \in B_\rho$ with $0<\rho<\rho(E).$  
\end{lemma}
\begin{proof}
	 Proof follows from the fact that by $(A^3)$,
	\begin{equation*}\label{R}
	\mathcal{R}^{E}(u)< q\frac{1}{|u|_{L^q}^q}\left(\frac{1}{2}H_\lambda(u)-C\right)< \frac{q}{|u|_{L^q}^q}\left(\frac{1}{2}\|u\|_1^2-C\right), \, \, u \in W \setminus \{0\}.
	\end{equation*}
	where $C \in (0,+\infty)$ does not depend on $u \in W$.
\end{proof}

\begin{corollary}\label{rho}
	Assume that $\rho<\rho(E)$. If $\hat{u} $ is a critical point of $\mathcal{R}^E_\rho(u)$ such that $\mathcal{R}^E_\rho(\hat{u})>0$, then $u $ is  a critical point of $\mathcal{R}^E(u)$ as well.
\end{corollary}
\begin{proof}
By Lemma \ref{lemNeg}, $\mathcal{R}^E_\rho(\hat{u})>0$ implies that $\|\hat{u}\|_W\geq \rho$. Therefore $\mathcal{R}^E_\rho(\hat{u})=\mathcal{R}^E(\hat{u})=\mu$ and $D\mathcal{R}^E(\hat{u})=0$. 
\end{proof}
 We say that a sequence $(u_{n}) \subset W$ is a Cerami sequence at the level $c \in \mathbb{R}$ of $\mathcal{R}^E$, in short $(Ce)_{c}$ sequence, whenever $\mathcal{R}^E(u_{n}) \rightarrow c$ and $(1 + \|u_n\|_W) \|D\mathcal{R}^E(u_{n})\| \rightarrow 0$ as $n \rightarrow \infty$. The functional $\mathcal{R}^E$ satisfies the Cerami condition at the level $c \in \mathbb{R}$, in short $(Ce)_{c}$ condition, whenever any $(Ce)_c$ sequence possesses a convergent subsequence. 
The definitions of the $(Ce)_{c}$ sequence and the $(Ce)_{c}$ condition for $\mathcal{R}^E_\rho$ are similar.
\begin{corollary}\label{Cc}
If $\rho<\rho(E)$,  then $\mathcal{R}^E_\rho(u)$ satisfies the $(Ce)$-condition (Cerami condition) at the level $c\in \mathbb{R}^+,$ if and only if $\mathcal{R}^E(u)$ does.
\end{corollary}

\begin{proof}
Assume $\mathcal{R}^E_\rho(u)$ satisfies the $(Ce)$-condition at $c>0$. 
Let $(u_n)$ be a $(Ce)$ sequence for $\mathcal{R}^E_\rho$ at $c$, i.e., $\mathcal{R}^E_\rho(u_{n}) \rightarrow c$ and $(1 + \|u_n\|_1) \|D\mathcal{R}^E_\rho(u_{n})\|_{\ast} \rightarrow 0$  as $n \rightarrow \infty$. 
	By Lemma \ref{lemNeg}, $\mathcal{R}^E_\rho(u)\leq 0$ as $u \in B_\rho$, whereas $\mathcal{R}^E_\rho(u)=\mathcal{R}^E(u)$ for $u \in W\setminus B_\rho$. Hence $\mathcal{R}^E_\rho(u)>0$ implies $\mathcal{R}^E(u)=\mathcal{R}^E_\rho(u)$ and $(\mathcal{R}^E)'(u)=(\mathcal{R}^E_\rho)'(u)$. Therefore $(u_n)$ is also $(Ce)_{c}$ sequence for $\mathcal{R}^E$. The proof of opposite statement is similar. 
\end{proof}

\section{On the properties of  $\mathcal{R}^E_\rho$}
Consider the sphere $S_r^\pm:=\{u \in W^\pm: \|u\|_1=r\}$, $r>0.$ 
Observe that for any given $\epsilon>0, $ $(A^1)\!\!-\!\!(A^2)$ imply the existense of $C(\epsilon)>0$ such that
\begin{equation}\label{CrescimentoG}
	G(x,s) \leq \frac{\epsilon}{2} |s|^2+C(\epsilon)|s|^\gamma,\;\;\;\;\;\;~~\forall s \in \mathbb{R}~\mbox{and a.e.}~x \in \Omega.
\end{equation}
This by the Sobolev inequalities implies
$$
\int G(x,u) dx \leq \frac{\epsilon}{2} |u|^2_{L^2}+C(\epsilon)|u|^\gamma_{L^\gamma}\leq 
\frac{\epsilon}{2} C_1\|u\|^2_1+C_2\|u\|^\gamma_1,~~u \in W,
$$
where $C_0, C_2 \in (0,+\infty)$ do not depend on $u \in W$.
 
\begin{proposition}\label{PositRHO} For any  $\lambda \in (\lambda_k, \lambda_{k+1})$,  there exist $E^k_{\lambda}>0$ and $r_\lambda^k>0$ such that 
$\inf\limits_{w\in S_{r_\lambda^k}^+}\mathcal{R}^E_\rho(w)>0$,  for any $E\in [0,E^k_\lambda)$, $\forall \rho \in (0,r_\lambda^k)$. 
\end{proposition}
\begin{proof} Note that  
$H_\lambda(w)\!=\|w\|_1^2$, $\forall w\in W^+$.
Take $\epsilon \in (0,1/C_1)$. Then by the Sobolev inequality and $(\ref{CrescimentoG})$ we have 
	$$
	\mathcal{R}^E(w)\;\;\geq\;\;q\frac{\frac{1}{2}\left(1-C_1\epsilon\right)\|w\|_1^2\;\;-\;\;C_2\|w\|_1^{\gamma}-E}{|w|_{L^q}^q}\;\;=\;\;q\frac{f\left(\|w\|_1\right)-E}{|w|_{L^q}^q}, ~~\forall w\in W^+,
	$$ 
where $\displaystyle{f(r)\!:=\!\frac{1}{2}(1-C_1\epsilon)r^2-C_2r^{\gamma}}$, and $C_1,C_2 \in (0,+\infty)$ do not depend on $u \in W$. Observe that  $f(r)$ attains its global maximum 
	$$
	E_\lambda^k:=f(r_\lambda^k)~~\mbox{ at }~~ r_\lambda^k:=\left[(1-C_1\epsilon)/(\gamma C_2)\right]^{1/(\gamma-2)}.
	$$ 
	Thus, for any $E\in [0,E^k_\lambda)$,
		$$
	\inf_{w\in S_{r_\lambda^k}^+}\mathcal{R}^E(w)\;\geq\;\inf_{w\in S_{r_\lambda^k}^+}q\frac{f(r_\lambda^k)-E}{|w|_{L^q}^q}\;\geq \;\;q\frac{E^k_\lambda-E}{S_q^q (r_\lambda^k)^q}\;\;=:\;\;\delta_{E}> 0, 
$$ 
which implies the proof, since $\mathcal{R}^E_\rho(w)=\mathcal{R}^E(w)>0$, $w \in S_{r_\lambda^k}^+$ if $\rho \in (0,r_\lambda^k)$.
\end{proof}
\begin{proposition}\label{p-infty}
For any $ u \in W\setminus 0$ and  $r>0$, there holds $\mathcal{R}^E(tu+v)\to -\infty$ as $t\to +\infty$ uniformly for $v \in B_r$.
		\end{proposition}

\begin{proof}
	Observe that  $(A^3)$ implies  $G(x,s)\!\geq\! c(x)|s|^{\alpha}\!>\!0$ a.e. in $\Omega$ and  $|s|\!\geq\! R_0$, with some $c\!\in\! L^\infty(\Omega)$. Hence  by $(A^2)$ we have $\displaystyle\; G(x,s)\;\geq\;c(x)|s|^{\alpha}-C_0$, $s \in \mathbb{R}$ and a.e $x \in \Omega$ for some constant $C_0 \in \mathbb{R}$. Note that $c^{1/\alpha}(x)\left(u(x)+v(x)/t\right)\to  c^{1/\alpha}(x)u(x)\;\; in \;\;L^{\alpha}(\Omega)$ uniformly in $v\in B_r$ as $t\to +\infty$. Indeed, using the Sobolev inequality we have
	$$
	\left|\int c(x)\left|u+\frac{v}{t}\right|^{\alpha}\,dx-\int c(x)|u|^{\alpha}\right|\,dx\leq \frac{1}{t}\int c(x)\left|v\right|^{\alpha}\,dx\leq \frac{1}{t}Cr^\alpha, ~~v \in B_r
	$$
	for some constant $C$ which does not depend on $v \in  B_r$. Thus, uniformly in $v\in B_r$,
	$$
	\lim\limits_{t\to\infty}\frac{1}{t^{\alpha}}\displaystyle \int G(x,tu+v)\,dx\;\;\;\geq\;\;\;\lim\limits_{t\to\infty}\left(\int c(x)\left|u+\frac{v}{t}\right|^{\alpha}dx\;\;\;-\;\;\;\frac{C_0|\Omega|}{t^{\alpha}}\right)\;\;=\;\;\int c(x)\left|u\right|^{\alpha}dx. \;\;
	$$
	This implies that $\lim\limits_{t\to\infty}\frac{1}{t^{\alpha}}\left(\|tu+v\|_W^2-\displaystyle\; \int G(x,tu+v)\,dx\right)\;\;\leq\;\;-\int c(x)\left|u\right|^{\alpha}dx\;\;<\;\;0,$  uniformly in $v\in B_r.$ Since $|u+v/t|^q_{L^q}\leq C\left(\;\|u\|^q_q+S^q_q r^q/t^q\right)$ and  $\alpha>2>q$, we conclude that uniformly for $v \in S_r$ there holds
	\begin{align}\label{struw1}\lim\limits_{t\to\infty}\mathcal{R}^E(tu+v)\;\leq\;\lim\limits_{t\to\infty}\frac{t^{\alpha-q}}{\left|\left|u+\frac{v}{t}\right|\right|_q^q}\left[\frac{\|tu+v\|_W^2}{2t^{\alpha}}-\int \frac{G(x,tu+v)}{t^{\alpha}}\,dx\right]\;\;=\;\;-\infty.
	\end{align}
\end{proof}
\begin{proposition}\label{convergence}
Assume that $E\in (0,E^k_{\lambda})$, $0<\rho<r_k^{\lambda}$. The functional $\mathcal{R}_{\rho}^E$ satisfies the $(Ce)$ condition at any level $\mu > 0$. 
\end{proposition}
\begin{proof} 
Assume that $(u_m)$ is a $(Ce)$ sequence, i.e.,  $\mu_m\!:=\!\mathcal{R}^E_{\rho}(u_m)\!\to \! \mu>0$ and
$\|D\mathcal{R}_{\rho}^E(u_m)\|_*(1+\|u_m\|_1)\! \to \! 0$ as $m \to +\infty$.
Since, $\mathcal{R}_{\rho}^E(u)\!=\!0$ on $B_{\rho}$ and $\mu\!>\!0$, we may assume that $\|u_m\|_W\!>\!\rho$ and $\mathcal{R}_{\rho}^E(u_m)\!=\!\mathcal{R}^E(u_m)$, $\forall m$.   Since by $(A^3)$, $  \int (g(x,u)u-\alpha G(x,u))>0$, $u \in W \setminus 0$, we have
\begin{align*}
	\mu_m(\alpha-q)+o(1)\|u_m\|_1(1+\|u_m\|_1)^{-1}\;=\;(\alpha-q) \mathcal{R}^E(u_m)&-D\mathcal{R}^E(u_m)(u_m)\geq\\
	&\frac{q}{|u_m|_{L^q}^q}\left[\frac{\alpha-2}{2}H_\lambda(u_m)-E\alpha\right].
\end{align*}
Hence $H_\lambda(u_m)\;\;\leq \;\;c_0\;(1+|u_m|_{L^q}^q)$, where $0<c_0<+\infty$ does not depend on $m=1,2,\ldots $, and therefore
\begin{equation}\label{estimativeCE2}
\|u_m\|_W^2\;\;\;\leq\;\;\;\lambda|u_m|_{L^2}^2\;\; +\;\;c_0\;(1+|u_m|_{L^q}^q),~m=1,2,\ldots.
\end{equation}
Thus, if $|u_m|_{L^2}$ is bounded, then $\|u_m\|_W$ is also bounded. 
If $|u_m|_{L^2}\to\infty$, then by \eqref{estimativeCE2} 
$\displaystyle{ \lim_{m\to\infty}\frac{H_\lambda(u_m)}{|u_m|_{L^2}^2}\;\leq\;0}$, and consequently, we get a contradiction
$$
0\;\;<\;\;c\;=\;\limsup\limits_{m\to\infty}\mathcal{R}^E(u_m)\;\;\;=\;\;\;\limsup\limits_{m\to\infty} q\frac{|u_m|_{L^2}^2}{|u_m|_{L^q}^q}\left[\frac{1}{2}\frac{H_\lambda(u_m)}{|u_m|_{L^2}^2}-\int\frac{ G(x,u_m)}{|u_m|_{L^2}^2}\,dx-\frac{E}{|u_m|_{L^2}^2}\right]\;\;\leq\;\;0.
$$
Thus $(u_m)$ is bounded and we may assume that
$u_m \rightharpoonup u$ weakly in $W$ and strongly in $L^r(\Omega)$, $r\in(1,2^*),$ as $m \to \infty$. In particular, this and $(A^2)$ give 
\begin{equation}\label{Gc}
	\int G(x,u_m) dx \to \int G(x,u) dx, ~~\|u_m\|^q_{q} \to \|u\|^q_{q} ~~\mbox{as}~~m \to +\infty.
\end{equation}By the convergence $\|D\mathcal{R}^E(u_m)\|\! \to\!  0$  we get $\left<\!D\mathcal{R}^E(u_m), u-u_m \!\right>\!\to\! 0$ as $m\to +\infty$. From the \eqref{Gc} and $S^+$ property  (see \cite{drabek})  we obtain that $ \left<-\Delta u_m, u-u_m \right> \to 0 $ and $u_m \to u$ strongly in $W$. 
\end{proof}
\begin{remark}\label{rem:nonzero} 
		Since  $\|u_m\|_1\!>\!\rho$, $\forall m$ and $\mathcal{R}^E_{\rho}(u_m)\!\to \! \mu \in (0,+\infty)$ it follows that  $|u_m|_{L^q}\not\to 0$.
\end{remark}

\subsection{The proof of Theorem \ref{thm1}}
Let $\lambda\!\in\! (\lambda_k,\lambda_{k+1})$, $E\!\in\!(0,E^k_\lambda)$  and $0<\rho<r_k^{\lambda}$. For $T>r_\lambda^k,$ take $\bar{u}^+ \in S^+_1$, and define 
\begin{align*}
	&B_o:=B_0(T)= \{u=t\bar{u}^+ + sv:~  v \in S^-_1,~ (0<t<T,~s= T)~\mbox{or}~(t \in \{0,T\},~0\leq s\leq  T)\},\\
	&B:= B(T)= \{u=t\bar{u}^+ + sv:~  v \in S^-_1,~0<t<T,~ 0\leq s\leq T\}, 
\end{align*} 
and
\begin{align*}
	&B_0^c:=B_0^c(T)\!:=\!\{u=t\bar{u}^++ sv:~\!  v\! \in\! S^-_1\!,~~ (0<t<T,~s= T)\},
	\\
	&B_0^d:=B_0^d(T)\!:=\!\{u=t\bar{u}^+ +  sv:~  v \in S^-_1\!,~~ t\! \in \{0,T\},~0\leq s\leq  T\}. 
	\end{align*}
Observe that $H_\lambda(u)= t^2\|\bar{u}^+\|_1^2-T^2\|v\|_1^2=(t^2-T^2)<0$ for $T>t$ and $u=t\bar{u}^+ +  Tv$, $u^+ \in S^+_1$, $v \in S^-_1$. This implies
$$
b^c(T):=\sup_{u \in B_0^c}\mathcal{R}^E_\rho(u)= \sup_{u \in B_0^c}\phi_\rho(\|u\|_1) \frac{\frac{1}{2}H_\lambda(u)-\int G(x,u)\,dx-E}{\frac{1}{q}|u|_{L^q}^q}\leq 0,~~\forall \rho>0. 
$$
 Proposition \ref{p-infty} implies that
$b^d(T):=\sup\limits_{u \in B_0^d}\mathcal{R}^E_\rho(u) \leq 0$, for sufficiently large $T>r_\lambda^k$. Thus,  by Proposition \ref{PositRHO},  for any $\rho\in (0,r_\lambda^k)$ and for sufficiently large $T>r_\lambda^k$ there holds
$$
b:=\sup_{u \in B_0}\mathcal{R}^E_\rho(u)\leq 0< \inf_{u \in S_{r_\lambda^k}^+}\mathcal{R}^E_\rho(u)=:a.
$$
Let $\lambda\!\in\! (\lambda_k,\lambda_{k+1})$, $E\!\in\!(0, E^k_\lambda)$ and $0<\rho<\min\{r_k^{\lambda},\rho(E)\}$. Consider
\begin{equation}\label{LVar}
\mu^k_\lambda(E):=\inf_{h \in \Gamma}\max_{u \in B}\mathcal{R}^E_\rho(h(u)),	
\end{equation}
where $\Gamma = \{h \in C(B; W) :~ h|_{B_0}=id_{B_0}\}$. 
By Propositions \ref{convergence} and  Corollary \ref{Cc} the functional $\mathcal{R}^E_\rho$ satisfies the $(Ce)$ condition at the level $c=\mu^k_\lambda(E)>0$. Moreover, $B_0\cap  S_{r_\lambda^k}^+ =\emptyset$ and $\{B_0,B\}$  links $S_{r_\lambda^k}^+$ in $W$ (see \cite{motreanu}), $
d(B_0, S_{r_\lambda^k}^+)>0$, and 
$S_{r_\lambda^k}^+$ is closed in $W$. 
 Hence, by the  Benci \& Rabinowitz Linking Theorem  \cite{BencRab}  for functionals satisfying $(Ce)$ condition (see \cite[Theorem 5.39]{motreanu}), there exists a nonzero critical point $u_{\mu^k_\lambda(E)}\in W$ of the functional $\mathcal{R}^E_\rho$ such that $\mathcal{R}^E_\rho(u_{\mu^k_\lambda(E)}) = \mu^k_\lambda(E)\geq a > 0$. Consequently, Corollary \ref{rho} and \eqref{R1} yields that $u_{\mu^k_\lambda(E)}$ is a weak solution of  \eqref{p} with $\mu=\mu^k_\lambda(E)$ and energy value $E$, i.e., $DE_{\mu^k_\lambda(E)}(u_{\mu^k_\lambda(E)})=0$, $E_{\mu^k_\lambda(E)}(u_{\mu^k_\lambda(E)})=E$.
Thus we have proved the first part of the theorem. 

\textit{Let us show (i). 
}Take $E_1>E_0>0$. It is not hard to see that the  sets $\{B_0,B\}$ and  the path sets $\Gamma$ in \eqref{LVar} can be taken the same for $E_1,E_0$ if $|E_1- E_0|$ is sufficiently small.
Note that
	$$
	\mathcal{R}^{E_1}_\rho(u)=\mathcal{R}^{E_0}_\rho(u)-\phi_\rho(\|u\|_1)\frac{E_1-E_0}{\int G(x,u)\,dx},~~\forall u \in W\setminus 0,
	$$
	and thus for sufficiently small $|E_1- E_0|$,
		\begin{equation*}
		\max_{u \in B}\mathcal{R}^{E_1}_\rho(h(u))= \max_{u \in B}\left(\mathcal{R}^{E_0}_\rho(h(u))-\phi_\rho(\|h(u)\|_1)\frac{E_1-E_0}{\int G(x,h(u))\,dx}\right)\leq \max_{u \in B}\mathcal{R}^{E_0}_\rho(h(u)), 
	\end{equation*}
$\forall h \in \Gamma$, and therefore, $
\mu^k_\lambda(E_1)=\inf_{h \in \Gamma}\max_{u \in B} \mathcal{R}^{E_1}_\rho(h(u))\leq  \inf_{h \in \Gamma}\max_{u \in B}\mathcal{R}^{E_0}_\rho(h(u))=\mu^k_\lambda(E_0)$.
\medskip

{\it Now we prove (ii).} Let $E=0$. Consider $\mathcal{R}^0(u)\equiv \mathcal{R}^E(u)|_{E=0}$. Using \eqref{CrescimentoG} and $1 < q < 2$ it is not hard to show that $\mathcal{R}^0(u) \to 0$ as $\|u\|_1 \to 0$. Consequently, by the continuation we can set  that $\mathcal{R}^0(0)=0$.

Assume that $\lambda<\lambda^1$. Then $W^-=\emptyset$ and $W^+\equiv W$. By Proposition \ref{p-infty}, one can find $u_1 \in W$ such that $\mathcal{R}^E(u_1)<0$. 
Let $E\in [0,E^0_\lambda)$ and $0<\rho<\min\{r_0^{\lambda},\rho(E)\}$. Observe that \eqref{LVar} can be rewritten as follows 
\begin{equation}\label{MPass}
\mu_\lambda^0(E):=\inf_{\gamma \in \Gamma}\max_{t \in [0,1]}\mathcal{R}^E_\rho(\gamma(t)),	
\end{equation}
where $\Gamma = \{\gamma \in C([0,1]; W) :~ \gamma(0)=0, ~\gamma(1)=u_1\} $, $0<\rho<\rho(E)$. Here we set $\mathcal{R}^0_\rho(u):=\mathcal{R}^0(u)$, $\rho>0$.  

Note that by the above, for any $E\in (0,E^0_\lambda)$ and $0<\rho<\min\{r_0^{\lambda},\rho(E)\}$, $\mu_\lambda^0(E)>0$ and there exists a critical point $u_{\mu_\lambda^0(E)} \in W\setminus 0$ of $\mathcal{R}^E(u)$ such that $DE_{\mu_\lambda^0(E)}(u_{\mu_\lambda^0(E)})=0$ and $E_{\mu_\lambda^0(E)}(u_{\mu_\lambda^0(E)})=E$.
As  in the proof of (i), from \eqref{MPass} it follows that $\mu_\lambda^0(E)$ is a nonincreasing function on $E \in [0,E^0_\lambda)$. Moreover,  $\mu_\lambda^0(E)\leq \mu_\lambda^0(0)<+\infty$, for any $E \in  (0,E^0_\lambda)$. Hence there exists  $\lim_{E \to 0}\mu_\lambda^0(E)=
\bar{\mu}_\lambda(0) \leq \mu_\lambda^0(0)$. Furthermore,  $\bar{\mu}_\lambda(0)>0$ since $\mu_\lambda^0(E)>0$, $E \in\!(0,E^k_\lambda)$ and  $\mu_\lambda^0(E)$ is a nonincreasing function. 
	
Since $D\mathcal{R}^E(u_{\mu_\lambda^0(E)})=0$ and $\mu_\lambda^0(E)\equiv \mathcal{R}^E(u_{\mu_\lambda^0(E)})\to\bar{\mu}_\lambda(0)>0$, any countable subset of $(u_{\mu_\lambda^0(E)})_{E \in (0,E^0_\lambda)}$ is  a $(Ce)$ sequence. Hence Proposition \ref{convergence} implies that there exists a sequences $u_{\mu_\lambda(E_m)}$, $m=1,2,\ldots $,  such that $\lim_{m\to +\infty}E_m=0$ and  $u_{\mu_\lambda(E_m)}$ convergences in $W$ to some point $u_{\bar{\mu}_\lambda(0)} \in W$ as $m\to +\infty$. Note that  $u_{\bar{\mu}_\lambda(0)} \neq 0$ (see Remark \ref{rem:nonzero} ), and therefore   $u_{\bar{\mu}_\lambda(0)}$ is a weak solution of \eqref{p} with $\mu=\bar{\mu}_{\lambda(0)}$. Moreover, 
$$
0=\lim_{m\to +\infty}E_m =\lim_{m\to +\infty}E_\mu(u_{\mu_\lambda(E_m)})= E_\mu(u_{\bar{\mu}_\lambda(0)}).
$$
Thus $u_{\bar{\mu}_\lambda(0)}$ is a solution with zero energy.
This ends the proof of the theorem.

\section{Acknowledgements}
The first author was supported by  RSF grant No. 22-21-00580.

%

\end{document}